\documentclass[11pt,reqno]{amsart}
\usepackage[utf8]{inputenc}
\usepackage{amsmath,amssymb,amsthm, mathrsfs}
\usepackage{indentfirst}
\usepackage{comment,relsize,braket}
\usepackage{enumerate,multicol}
\usepackage[numbers]{natbib}
\usepackage{graphicx}
\usepackage{color}
\usepackage[all]{xy}
\usepackage{fullpage}
\usepackage{tikz}
\usepackage[pagebackref,colorlinks, linkcolor=blue]{hyperref}
\usepackage[top=1in, bottom=1in, left=.75in, right=.75in]{geometry}

\newtheorem{Theorem}{Theorem}[section]
\newtheorem{Lemma}[Theorem]{Lemma}
\newtheorem{Corollary}[Theorem]{Corollary}

\newtheorem{Remark}[Theorem]{Remark}

\newtheorem{Example}[Theorem]{Example}
\newtheorem{Examples}[Theorem]{Examples}
\newtheorem{Definition}[Theorem]{Definition}
\newtheorem{Conjecture}[Theorem]{Conjecture}
\newtheorem{Question}[Theorem]{Question}
\numberwithin{equation}{section}

\everymath{\displaystyle}

\def\QQ{{\mathbb Q}} \def\NN{{\mathbb N}} \def\ZZ{{\mathbb Z}}

\def\frk{\mathfrak}  \def\pp{{\frk p}}
   
 \def\nn{{\frk n}} \def\mm{{\frk m}}
\def\opn#1#2{\def#1{\operatorname{#2}}} 
\opn\chara{char} \opn\length{\ell} \opn\pd{pd} \opn\rk{rk}
\opn\projdim{proj\,dim} \opn\injdim{inj\,dim}
\opn\rank{rank} \opn\depth{depth} \opn\grade{grade} 
\opn\hei{ht} \opn\embdim{emb\,dim}\opn\codim{codim}
\opn\Tr{Tr} \opn\bigrank{big\,rank}
\opn\superheight{superheight} \opn\lcm{lcm}
\opn\rdim{rdim} \opn\trdeg{tr\,deg} \opn\reg{reg}  \opn\lreg{lreg} 
\opn\ini{in} \opn\lpd{lpd} \opn\size{size} \opn{\mult}{mult}
\opn\div{div} \opn\Div{Div} \opn\cl{cl} \opn\Cl{Cl}
\opn\Spec{Spec} \opn\Supp{Supp} \opn\supp{supp} 
\opn\Sing{Sing} \opn\Ass{Ass} \opn\Min{Min}
\opn\Proj{Proj} \opn{\Max}{Max} \opn{\Assh}{Assh}
\opn\Ann{Ann} \opn\Rad{Rad} \opn\Soc{Soc}
\opn\Syz{Syz} \opn\Im{Im} \opn\Ker{Ker} \opn\Coker{Coker}
\opn\Am{Am} \opn\Hom{Hom} \opn\tor{Tor} \opn\Ext{Ext}
\opn\End{End} \opn\Aut{Aut} \opn\id{id}

\opn\nat{nat} \opn\pff{pf} 
\opn\Pf{Pf} \opn\GL{GL} \opn\SL{SL} \opn\mod{mod} \opn\ord{ord}
\opn\Gin{Gin} \opn\Hilb{Hilb}
\opn\adeg{adeg} \opn\std{std}\opn\ip{infpt}
\opn\Pol{Pol} \opn\sat{sat} \opn\Var{Var}
\opn\aff{aff} \opn\con{conv} \opn\relint{relint} \opn\st{st}
\opn\lk{lk} \opn\cn{cn} \opn\core{core} \opn\vol{vol}
\opn\link{link} \opn\star{star}
\opn\gr{gr}

\opn\inn{in}

\title[]{On the Hilbert-Samuel Coefficients of Frobenius Powers \\ of an ideal}

\author[]{Arindam Banerjee}
\address{Indian Institute of Technology Kharagpur, Kharagpur, India}
\email{123.arindam@gmail.com}

\author[]{Kriti Goel}
\address{University of Utah, Utah, USA 84112}
\email{kritigoel.maths@gmail.com}

\author[]{J. K. Verma}
\address{Indian Institute of Technology Bombay, Mumbai, India 400076}
\email{jkv@math.iitb.ac.in}

\thanks{The second author is supported by Fulbright-Nehru Postdoctoral Research Fellowship.}
\thanks{{\it Key words and phrases}: Hilbert-Samuel coefficients, Hilbert-Kunz function, Hilbert-Kunz multiplicity, Frobenius power, face ring of simplicial complex, generalized Hilbert-Kunz function, generalized Hilbert-Kunz multiplicity.}
\thanks{{\it 2010 AMS Mathematics Subject Classification:} 13A30, 13C14, 13C15, 13D40, 13F55.}

\begin{document}

\begin{abstract}
	We provide suitable conditions under which the asymptotic limit of the Hilbert-Samuel coefficients of the Frobenius powers of an $\mm$-primary ideal exists in a Noetherian local ring $(R,\mm)$ with prime characteristic $p>0.$ This, in turn, gives an expression of the Hilbert-Kunz multiplicity of powers of the ideal. We also prove that for a face ring $R$ of a simplicial complex and an ideal $J$ generated by pure powers of the variables, the generalized Hilbert-Kunz function $\ell(R/(J^{[q]})^k)$ is a polynomial for all $q,k$ and also give an expression of the generalized Hilbert-Kunz multiplicity of powers of $J$ in terms of Hilbert-Samuel multiplicity of $J.$ We conclude by giving a counter-example to a conjecture proposed by I. Smirnov which connects the stability of an ideal with the asymptotic limit of the first Hilbert coefficient of the Frobenius power of the ideal.
\end{abstract}

\maketitle

\section{Introduction}

This article is inspired by the work of I. Smirnov in \cite{ilya}, where the author studies the Hilbert-Kunz multiplicity of powers of an ideal. Let $(R,\mm)$ be a $d$-dimensional Noetherian local ring of prime characteristic $p$ and let $I$ be an $\mm$-primary ideal. The $q^{th}$-Frobenius power of $I$ is the ideal $I^{[q]}=(x^q\mid x\in I)$ where $q=p^e$ for $e\in \NN.$ The function $e \mapsto \ell_R(R/I^{[p^e]})$ is called the {\it Hilbert-Kunz function} of $R$ with respect to $I$ and was first considered by E. Kunz in \cite{kunz69}. In \cite{monsky}, P. Monsky showed that this function is of the form 
\[ \ell_R(R/I^{[q]}) = e_{HK}(I,R) q^{d} + O(q^{d-1}), \]
where $e_{HK}(I,R)$ is a positive real number called the {\it Hilbert–Kunz multiplicity} of $R$ with respect to $I.$ We write $e_{HK}(R) := e_{HK}(\mm,R)$ and $e_{HK}(I) := e_{HK}(I,R).$ The Hilbert-Kunz multiplicity of powers of an ideal was first considered by D. Hanes in \cite[Theorem 3.2]{hanesNotes}. He proved that 
\[ \ell\left(\frac{R}{(I^{[q]})^k}\right) = \left( \frac{e_0(I)}{d!}k^d + O(k^{d-1}) \right)q^d, \]
where $e_0(I)$ denotes the Hilbert-Samuel multiplicity of $R$ with respect to $I.$ 
Recall that the {\it Hilbert-Samuel function} $H_I(n)$ of $R$ with respect to $I$ is defined as $H_I(n) = \ell_R(R/I^n).$ It is known that $H_I(n)$ is a polynomial function of $n$ of degree $d$, for large $n.$ In particular, there exists a polynomial $P_I(x) \in \QQ[x],$ called the Hilbert-Samuel polynomial, such that $H_I(n) = P_I(n)$ for all large $n.$ Write 
\[ P_I(x) = e_0(I) \binom{x+d-1}{d} - e_1(I) \binom{x+d-2}{d-1} + \cdots + (-1)^d e_d(I), \]
where $e_i(I)$ for $i = 0, 1, \ldots, d$ are integers, called the {\it Hilbert-Samuel coefficients of $I.$} The leading coefficient $e_0(I)$ is called the multiplicity of $I$ and $e_1(I)$ is called the Chern number of $I.$ 

In \cite{trivedi}, Trivedi proved that if $R$ is a standard graded ring of dimension $d$ over a perfect field of characteristic $p>0,$ and $I$ is a homogeneous ideal of finite colength and generated in the same degree, then
\begin{align*}
	\lim\limits_{k \rightarrow \infty} \frac{e_{HK}(I^k) - e_0(I^k)/d!}{k^{d-1}} = \frac{e_0(I)}{2(d-2)!} - \lim\limits_{q \rightarrow \infty} \frac{e_1(I^{[q]})}{q^d(d-1)!}
\end{align*}
and the last limit exists. The result was proved in full generality by Smirnov in \cite[Proposition 2.5]{ilya}. He proved that if $(R,\mm)$ is a Noetherian local ring and $I$ is an $\mm$-primary ideal, then
\[ e_{HK}(I^k) = e_0(I) \binom{k+d-1}{d} - \lim\limits_{q \rightarrow \infty} \frac{e_1(I^{[q]})}{q^d} \binom{k+d-2}{d-1} + o(k^{d-1}). \]
The above equation motivates one to ask if other Hilbert coefficients may also be appear in the expression.

One of the aims of this article is to provide some suitable conditions under which the following questions, proposed by Smirnov in \cite{ilya}, can be answered in the affirmative. Let $(R,\mm)$ be a Noetherian local ring of dimension $d$ with prime characteristic $p>0.$ Let $I$ be an $\mm$-primary ideal. Put $q = p^e,$ for some $e \in \NN.$

\begin{Question}\label{q1}
	Does the limit $\lim\limits_{q \rightarrow \infty} e_i(I^{[q]})/q^d$ exist for all $i \geq 1$ $?$
\end{Question}

\begin{Question}\label{q2}
	Do we have for large $k,$ that 
	\[  e_{HK}(I^k) = \sum_{i=0}^{d} (-1)^i \binom{k+d-1-i}{d-i} \lim\limits_{q \rightarrow \infty} \frac{e_i(I^{[q]})}{q^d} \ ? \]
\end{Question}

Recall that for an ideal $I$ in a ring $R,$ a reduction of $I$ is an ideal $J \subseteq I$ such that $JI^n = I^{n+1}$, for all large $n.$ A minimal reduction of $I$ is a reduction of $I$ minimal with respect to inclusion. For a minimal reduction $J$ of $I$, we set $r_J(I) = \min \{ n \mid  I^{m+1} = JI^m \text{ for all } m \geq n \}.$ The reduction number $r(I)$ of $I$ is defined as 
\[ r(I) = \min\{ r_J(I) \mid J \text{ is a minimal reduction of } I \}. \]
An ideal with reduction number one is called a stable ideal.

It is easy to check that the questions \ref{q1} and \ref{q2} have affirmative answers when $R$ is a Cohen-Macaulay local ring and $I$ is a parameter ideal, or when $R$ is a Cohen-Macaulay local ring of dimension $d \geq 2$ and $I$ is a stable ideal (see \cite[Proposition 1.7, Theorem 1.8]{watanabeYoshidaDim2}). A class of rings satisfying the latter assumption includes rings with minimal multiplicity, taking $I = \mm.$ Recall that $R$ has minimal multiplicity if $\text{emb}(R) - \dim R +1 = e_0(R),$ where $\text{emb}(R)$ denotes the embedding dimension of $R.$

In section 2, we prove the following results which answer the above questions in the affirmative under suitable assumptions.

\begin{Theorem}
	Let $(R,\mm)$ be a Buchsbaum ring of dimension $d$ with prime characteristic $p$ and let $I$ be an ideal generated by a system of parameters. Then for all $i = 1,\ldots,d,$ $\lim\limits_{q \rightarrow \infty} e_i(I^{[q]})/q^d=0$ and for all $k \geq 1,$
	\begin{align*}
		e_{HK}(I^k) = e_0(I) \binom{k+d-1}{d}.
	\end{align*}
\end{Theorem}

\begin{Theorem}
	Let $(R,\mm)$ be a Cohen-Macaulay local ring of dimension $d\geq 1$  and prime characteristic $p.$ Let $I$ be an $\mm$-primary ideal. 
	Suppose that $r=r(I)$ and  $\depth G(I^{[q]})\geq d-1$  for all large $q.$ Then
	\begin{enumerate}[{\rm(1)}]
		\item For all $i=1,2,\ldots, d$ the limit $L_i(I):=\lim_{q\to \infty} e_i(I^{[q]})/q^d$ exists and
		\[L_i(I)=\lim_{q\to \infty} \frac{e_i(I^{[q]})}{q^d}=\sum_{n=i}^r\binom{n-1}{i-1} \left[e_0(I)-\sum_{j=0}^d(-1)^j\binom{d}{j}e_{HK}(I^{n-j})\right]
		.\]
		\item For all $n \geq r-d+1,$
		\[\ell\left(\frac{R}{(I^{[q]})^n}\right)=\sum_{i=0}^d (-1)^ie_i(I^{[q]})\binom{n+d-1-i}{d-i} \]
		and hence for all $n\geq r-d+1,$ we have
		\[ e_{HK}(I^n)=\sum_{i=0}^d(-1)^iL_i(I)\binom{n+d-1-i}{d-i}.\] 
	\end{enumerate}
\end{Theorem}

The generalized Hilbert-Kunz function was introduced by Aldo Conca in \cite{conca}. Let $(R,\mm)$ be a $d$-dimensional Noetherian local  (resp. standard graded)  ring with maximal (resp. maximal homogeneous) ideal $\mm$  and $I$ be an $\mm$-primary (resp. a graded $\mm$-primary) ideal. Fix a set of generators of $I$, say $I=(a_1, a_2,\ldots, a_g).$ We choose these as homogeneous elements in case $R$ is a graded ring. Define the $s^{th}$ Frobenius power of $I$ to be the ideal $I^{[s]}=(a_1^s, a_2^s, \ldots, a_g^s).$ The generalized Hilbert-Kunz function of $I$ is defined as $HK_I(s)=\ell(R/I^{[s]}).$ The generalized Hilbert-Kunz multiplicity is defined as $\lim\limits_{s \rightarrow \infty} HK_I(s)/s^d$, whenever the limit exists and is denoted by $e_{HK}(I)$.

In sections $3$ and $4,$ we compute the generalized Hilbert-Kunz function of powers of certain monomial ideals in face rings of simplicial complexes. In particular, these results are characteristic independent. 

\begin{Theorem}  \label{HKSC}
	Let $K$ be a field and $R = K[x_1,\ldots,x_r]/I_{\Delta}$ be a $d$-dimensional face ring of a simplicial complex $\Delta.$ Let $J = (x_1^{v_1},\ldots,x_r^{v_r})R,$ where $v_i >0$ for all $i.$ Then $\lim\limits_{q \rightarrow \infty} e_i(J^{[q]})/q^d =0$ for all $i=1,\ldots,d$ and for all $k \geq 1,$ the generalized Hilbert-Kunz multiplicity 
	\[ e_{HK}(J^k) = e_0(J) \binom{k+d-1}{d}. \]
\end{Theorem}

Smirnov, in the same paper, also proposed the following conjecture.
\begin{Conjecture}
	Let $(R,\mm)$ be a Cohen-Macaulay local ring. Then an $\mm$-primary ideal $I$ is stable if and only if $\lim\limits_{q \rightarrow \infty}	e_1(I^{[q]})/q^d = e_0(I) - e_{HK}(I).$
\end{Conjecture}

It was perhaps motivated by a result of Huneke and Ooishi, which characterized the stability of an ideal using the Northcott's inequality. In section 4, we give a counter example to the conjecture.
In fact, using Theorem \ref{HKSC} above, it follows that for any ideal $J$ of a face ring $R$ of the form $(x_1^{v_1},\ldots,x_r^{v_r})R$ where $v_i >0$ for all $i,$ $\lim\limits_{q \rightarrow \infty}	e_1(J^{[q]})/q^d = e_0(J) - e_{HK}(J).$ This means that one side of the conjecture is always true for such ideals. In particular, the conjecture is false whenever $J$ is not a stable ideal. The following example illustrates one such case.


\begin{Example}{\rm 
	Let $\Delta$ be the simplicial complex 
	
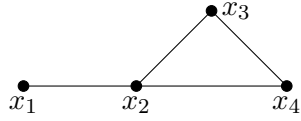
\begin{center}
		\begin{tikzpicture}
			\draw (0,0) -- (1.5,0);
			\draw (1.5,0) -- (2.5,1) -- (3.5,0) -- cycle;
			\filldraw (0,0) circle (2pt) node[anchor=north]{$x_1$};		
			\filldraw (1.5,0) circle (2pt) node[anchor=north]{$x_2$};
			\filldraw (2.5,1) circle (2pt) node[anchor=west]{$x_3$};
			\filldraw (3.5,0) circle (2pt) node[anchor=north]{$x_4$};
		\end{tikzpicture}
	\end{center}
	Let $K$ be a field. Then $R = K[x_1,x_2,x_3,x_4]/((x_3,x_4) \cap (x_1,x_3) \cap (x_1,x_4) \cap (x_1,x_2))$ is the face ring of $\Delta.$ Observe that $R$ is a $2$-dimensional Cohen-Macaulay ring.
	 Let $\nn = (x_1,x_2,x_3,x_4)$ denote the maximal ideal of $R.$ Using Theorem \ref{HKSC}, it follows that $4 = e_0(\nn) = e_{HK}(\nn)$ and $\lim\limits_{q \rightarrow \infty} e_1(\nn^{[q]})/q^2=0.$ We prove that $\nn$ is not stable in section 4.
	
}\end{Example}

\section{Hilbert-Kunz multiplicity of powers of ideals}
In this section we answer Smirnov's questions,  mentioned in the introduction, in  the affirmative in certain cases. We show that both the questions have answers in the affirmative for parameter ideals in  Buchsbaum local rings. We then  consider the questions  for $\mm$-primary ideals $I$  in Cohen-Macaulay local rings $(R,\mm)$ of dimension $d$
subject to the condition  $\depth G(I^{[q]})\geq d-1.$ This is a strong condition. However, we show that it is satisfied in several examples (in sections 3 and 4).
We begin by considering the case of Buchsbaum local rings.

\begin{Theorem}
	Let $(R,\mm)$ be a Buchsbaum ring of dimension $d$ with prime characteristic $p$ and let $I$ be a parameter ideal. Then for all $i = 1,\ldots,d,$ $\lim\limits_{q \rightarrow \infty} e_i(I^{[q]})/q^d=0$ and for all $k \geq 1,$
	\begin{align*}
		e_{HK}(I^k) = e_0(I) \binom{k+d-1}{d}.
	\end{align*}
\end{Theorem}

\begin{proof}
	Since $I$ is a parameter ideal, then so is $I^{[q]}$ for all $q.$ Therefore, using \cite[Corollary 4.2]{trungSuperficial}, we obtain that for all $i = 1,\ldots,d,$
	\begin{align}   \label{thm1:eq1}
		e_i(I^{[q]}) = (-1)^i \sum_{j=0}^{d-i} \binom{d-i-1}{j-1} \ell(H^i_{\mm}(R)),
	\end{align}
	where $\binom{d-i-1}{-1} = 0$ if $i \neq d$ and $\binom{-1}{-1} = 1.$ Also, for all $k \geq 1,$
	\begin{align}  \label{thm1:eq2}
		\ell\left(\frac{R}{I^{[q]k}}\right) = e_0(I^{[q]}) \binom{k+d-1}{d} + \sum_{i=1}^{d} e_i(I^{[q]}) \binom{k+d-1-i}{d-i}. 
	\end{align}
	From \eqref{thm1:eq1}, it follows that for all $i=1,\ldots,d,$ $e_i(I^{[q]})$ is independent of $q$ and hence using \eqref{thm1:eq2}, we get the required result.
\end{proof}

The next result answers Smirnov's questions for $\mm$-primary ideals  $I$ in Cohen-Macaulay local rings of prime characteristic $p$ and dimension $d\geq 1$ and $\depth G(I^{[q]})\geq d-1$ for all large $q.$  Recall that for a function $f: \ZZ\to \ZZ$,  $\Delta(f)(n)=f(n)-f(n-1).$ Define $\Delta^t(f)=\Delta(\Delta^{t-1}(f))$ for all $t\geq 2.$ We shall use the following results due to S. Huckaba and T. Marley.

\begin{Theorem}[{\cite{huckaba,huckabaMarley}}]    \label{hm}
	Let $(R,\mm)$ be a Cohen-Macaulay local ring of dimension $d\geq 1.$ Let $I$ be an $\mm$-primary ideal with a minimal reduction $J=(a_1,a_2, \ldots, a_d).$ Suppose that $r=r(I)$ and  $\depth G(I)\geq d-1.$ Then for all $n\in \ZZ,$
	\begin{enumerate}[{\rm(1)}]
		\item $\Delta^d[P_I(n)-H_I(n)]=\ell(I^{n}/JI^{n-1})$.
		\item $\Delta^d P_I(n)=e_0(I)$ and $\Delta^d(H_I(n)) = \sum_{j=0}^d (-1)^j \binom{d}{j} \ell(R/I^{n-j}).$
		\item $e_i(I)=\sum_{n=i}^r\binom{n-1}{i-1}\ell(I^n/JI^{n-1})$ for $i=1,2,\ldots, d.$
	\end{enumerate}
\end{Theorem} 
 
\begin{Theorem}
	Let $(R,\mm)$ be a Cohen-Macaulay local ring of dimension $d\geq 1$  and prime characteristic $p.$ Let $I$ be an $\mm$-primary ideal with a minimal reduction $J=(a_1,a_2, \ldots, a_d).$ Suppose that $r=r(I)$ and  $\depth G(I^{[q]})\geq d-1$  for all large $q.$ Then
	\begin{enumerate}[{\rm(1)}]
		\item For all $i=1,2,\ldots, d$ the limit $L_i(I):=\lim_{q\to \infty} e_i(I^{[q]})/q^d$ exists and
		\[L_i(I)=\lim_{q\to \infty} \frac{e_i(I^{[q]})}{q^d}=\sum_{n=i}^r\binom{n-1}{i-1} \left[e_0(I)-\sum_{j=0}^d(-1)^j\binom{d}{j} e_{HK}(I^{n-j})\right]
		.\]
		\item For all $n \geq r-d+1,$
		\[\ell\left(\frac{R}{(I^{[q]})^n}\right)=\sum_{i=0}^d (-1)^ie_i(I^{[q]})\binom{n+d-1-i}{d-i} \]
		and hence for all $n\geq r-d+1,$ we have
		\[ e_{HK}(I^n)=\sum_{i=0}^d(-1)^iL_i(I)\binom{n+d-1-i}{d-i}.\] 
	\end{enumerate}
\end{Theorem}

\begin{proof}
	Since $JI^r=I^{r+1},$ it follows that $J^{[q]}(I^{[q]})^r=(I^{[q]})^{r+1}.$ Hence $r(I^{[q]})\leq r$ for all $q.$ Since $\depth G(I^{[q]})\geq d-1$ for all large $q,$ $P_{I^{[q]}}(n)=H_{I^{[q]}}(n)$ for all $n\geq r-d+1.$ For all $i=1,2,\ldots, d,$ by Theorem \ref{hm}(3),
	\[ e_i(I^{[q]}) = \sum_{n=i}^r \binom{n-1}{i-1} \ell\left(\frac{(I^{[q]})^n}{J^{[q]}(I^{[q]})^{n-1}}\right).\]
	Let us find $\beta_I(n):=\ell(I^n/JI^{n-1}).$ By Theorem \ref{hm}, for all $n\in \ZZ,$ $\Delta^d(P_I(n)-H_I(n))=\ell(I^n/JI^{n-1}).$ Therefore
	\[ \beta_I(n) = e_0(I) - \Delta^d(H_I(n)) = e_0(I) - \sum_{j=0}^d (-1)^j \binom{d}{j} \ell(R/I^{n-j}).\]
	Since $\depth G(I^{[q]})\geq d-1$ for large $q,$ we have
	\[ \beta_{I^{[q]}}(n) = q^de_0(I) - \Delta^d(H_{I^{[q]}}(n)) = q^d e_0(I) - \sum_{j=0}^d (-1)^j \binom{d}{j} \ell(R/(I^{[q]})^{n-j}).\]
	Divide by $q^d$ and then take the limit to get
	\[ \lim_{q \to \infty} \frac{\beta_{I^{[q]}}(n)}{q^d} = e_0(I) - \sum_{j=0}^d (-1)^j \binom{d}{j} e_{HK}(I^{n-j}) .\]
	Therefore we have
	\begin{align}\label{L_i(I)} 
		L_i(I)=\lim_{q\to \infty} \frac{e_i(I^{[q]})}{q^d} = \sum_{n=i}^r \binom{n-1}{i-1} \left[e_0(I) - \sum_{j=0}^d (-1)^j \binom{d}{j} e_{HK}(I^{n-j})\right].
	\end{align}

Since $r(I^{[q]})\leq r,$ for all $n\geq r-d+1,$ $H_{I^{[q]}}(n)=P_{I^{[q]}}(n).$ Hence for all $n\geq r-d+1,$
\[\ell(R/(I^{[q]})^n)=\sum_{i=0}^d (-1)^ie_i(I^{[q]})\binom{n+d-1-i}{d-i}.\]
Divide by $q^d$ and take the limit to see that for all $n\geq r-d+1,$
\[e_{HK}(I^n)=\sum_{i=0}^d(-1)^iL_i(I)\binom{n+d-1-i}{d-i}.\]
\end{proof}

\begin{Corollary} 
	Let $(R,\mm)$ be a one-dimensional Cohen-Macaulay local ring of prime characteristic $p.$ Let $J=(a)$ be a minimal reduction of an $\mm$-primary ideal $I$ with $JI^r=I^{r+1}.$ Then $L_1(I)=0.$ 
\end{Corollary}

\begin{proof}
Put $i=1$ in the equation \ref{L_i(I)} to get 
\begin{align*}
L_1(I)
= \lim_{q\to \infty} \frac{e_1(I^{[q]})}{q^d}
= & \ \sum_{n=1}^r \left[ e_0(I)-\sum_{j=0}^1(-1)^j e_{HK}(I^{n-j})\right]\\
= & \ re_0(I)-\sum_{n=1}^r\left[e_{HK}(I^n)-e_{HK}(I^{n-1})\right]\\
= & \ re_0(I)-e_{HK}(I^r)\\
= & \ 0.
\end{align*}
\end{proof}

\begin{Corollary}  \label{dim2,r}
Let $(R,\mm)$ be a $2$-dimensional Cohen-Macaulay local ring of characteristic $p$. Let $J=(a, b)$ be a minimal reduction of an $\mm$-primary ideal $I$ with $r_J(I)=r.$
Let   $\depth G(I^{[q]})\geq 1$  for all large $q.$ Then
\begin{align*}
L_1(I)=& \ re_0(I)-e_{HK}(I^r)+e_{HK}(I^{r-1})\\
L_2(I)=& \ \binom{r}{2} e_0(I)-(r-1)e_{HK}(I^r)+re_{HK}(I^{r-1})\\
e_{HK}(I^n)=& \ e_0(I)\binom{n+1}{2}-L_1(I)n+L_2(I) \mbox{ for all } n\geq r-1.
\end{align*}
\end{Corollary}

\begin{proof}
Put $d=2$ and $i=1$ in \ref{L_i(I)} to get
\begin{align*}
L_1(I)=&\sum_{n=1}^r\left[e_0(I)-\sum_{j=0}^2 (-1)^j\binom{2}{j}e_{HK}(I^{n-j})\right]\\
=& \ re_0(I)-\sum_{n=1}^r \left[ e_{HK}(I^n)-2e_{HK}(I^{n-1})+e_{HK}(I^{n-2}) \right]\\
=& \ re_0(I)-e_{HK}(I^r)+e_{HK}(I^{r-1}).
\end{align*}
Now put $d=2$ and $i=2$ in \ref{L_i(I)} to get
\begin{align*}
L_2(I)
=& \ \sum_{n=2}^r(n-1)\left[ e_0(I)-\sum_{j=0}^2(-1)^j\binom{2}{j}e_{HK}(I^{n-j})\right]\\
=& \ \binom{r}{2} e_0(I)- \sum_{n=2}^r(n-1) \left[ e_{HK}(I^n) - 2e_{HK}(I^{n-1}) + e_{HK}(I^{n-2}) \right]\\
=& \ \binom{r}{2} e_0(I)-(r-1)e_{HK}(I^r)+re_{HK}(I^{r-1}).
\end{align*}
	From \cite[Theorem 2.15]{marleyThesis}, it also follows that the postulation number of $I^{[q]},$ $n(I^{[q]}) = r(I^{[q]}) - 2 \leq r-2.$ Thus, for all $n \geq r-1,$
	\begin{align*}
		\ell\left(\frac{R}{I^{n[q]}}\right) = e_0(I^{[q]}) \binom{n+1}{2} - e_1(I^{[q]}) n + e_2(I^{[q]}).
	\end{align*}
	Divide by $q^2$ and take the limit to see that for all $n\geq r-1$ 
	\[e_{HK}(I^n)=e_0(I)\binom{n+1}{2}-L_1(I)n+L_2(I).\]
\end{proof}

\begin{Example}{\rm
	Let $R = K[[x,y,z]]/(x^3-y^3z)$, where $K$ is a field of characteristic $3.$ Let $\mm = (x,y,z)$ denote the maximal ideal of $R.$ Then $I = (y,z)$ is a minimal reduction of $\mm$ and the reduction number of $\mm,$ $r(\mm) = 2.$ Set $q = 3^e,$ $e \in \NN.$ Then $r(\mm^{[q]}) \leq 2.$ We show that $G(\mm^{[q]})$ is Cohen-Macaulay for all large $q.$ Using Valabrega-Valla's result (\cite{VV}), it is sufficient to show that for all $n \geq 1,$ 
	\[ (y^q, z^q) \cap (\mm^{[q]})^n = (y^q,z^q) (\mm^{[q]})^{n-1}. \]
	Since $r(\mm^{[q]}) \leq 2,$ it is enough to show that $(y^q,z^q) \cap (\mm^{[q]})^2 = (y^q,z^q) \mm^{[q]}.$ Consider
	\begin{align*}
		(y^q,z^q) \cap (\mm^{[q]})^2 = (y^q,z^q) \cap (x^{2q}, x^qy^q, x^qz^q, y^{2q}, y^qz^q, z^{2q}).
	\end{align*}
	Let 
	\[ ay^q + bz^q = c x^{2q} + d x^qy^q + e x^qz^q + f y^{2q} + g y^qz^q + h z^{2q}, \]
	for some $a,b,c,d,e,f,g,h \in R.$ Then
	\[ y^q \big( a - dx^q - fy^q - gz^q \big) = z^q \big(-b + ex^q + hz^q \big) + cx^{2q}.  \]
	Since $x^3 = y^3z,$ it follows that 
	\[ x^{2q} = (x^3)^{2.3^{e-1}} = (y^3z)^{2.3^{e-1}} = y^{2q}z^{2q/3}. \]
	Therefore,
	\[ a - dx^q - fy^q - gz^q \in (z^q,y^{2q}) : y^q = (z^q,y^q) \]
	and hence, $a \in \mm^{[q]}.$ Similarly, $b \in \mm^{[q]}.$ This proves that $G(\mm^{[q]})$ is Cohen-Macaulay for all large $q.$ 
	
	Note that $e_0(\mm) = 3$ and $e_{HK}(\mm) = 3$ (see \cite[Theorem 3.1]{conca}). In order to compute $e_{HK}(\mm^2)$, we compute $\ell(K[x,y,z]/(x^3-y^3z,x^{2q},y^{2q},z^{2q},x^qy^q,x^qz^q,y^qz^q)).$ It is easy to check that $\{ x^3-y^3z, y^{2q}, y^qz^q, z^{2q} \}$ is a Gr\"obner basis of the ideal $(x^3-y^3z)+\mm^{2[q]}$ (This involves assigning different weights to the variables, $\deg x = 2,$ $\deg y = 1,$ $\deg z = 3$, in order to make the polynomial $x^3-y^3z$ homogeneous). Hence, the Hilbert-Kunz function 
	\[ \ell\left(\frac{K[[x,y,z]]}{(x^3-y^3z)+\mm^{2[q]}}\right) = \ell\left(\frac{K[x,y,z]}{(x^3,y^{2q},y^qz^q,z^{2q})}\right) = 9q^2. \]
	Therefore, $e_{HK}(\mm^2) = 9.$ Using Corollary \ref{dim2,r}, we obtain that for all $k \geq 1,$
	\begin{align*}
		e_{HK}(\mm^k) = e_0(\mm) \left[\binom{k+1}{2} - 2k + 1 \right] - (k-2) \ e_{HK}(\mm) + (k-1) \ e_{HK}(\mm^2).
	\end{align*}
}\end{Example}

\section{Face rings of simplicial complexes}

 Let $\Delta$ be a $(d-1)$-dimensional simplicial complex on $r$ vertices. Let $K$ be a field and $R = K[\Delta]$ be the face ring of simplicial complex $\Delta.$ Then $R \simeq S/I$ is a $d$-dimensional ring, where $S=k[x_1,\ldots,x_r]$ and $I = I_{\Delta}$ is the Stanley-Reisner ideal. The ideal $I$ can be viewed as intersection of face ideals:
\[ I = \bigcap_{\substack{F \in \Delta \\ F \text{ is a facet}}} \pp_{F},  \]
where $\pp_F = (x_i \mid i \notin F).$ Let $\nn = \mm/I$ denote the maximal homogeneous ideal of $R$, where $\mm$ denotes the maximal homogeneous ideal of $S.$ Set $J = (x_i^{v_i} \mid v_i > 0 \text{ for all } 1 \leq i \leq r).$ We compute the generalized Hilbert-Kunz function and Hilbert-Kunz multiplicity of  powers of $J$ and in particular, of $\nn.$ We prove a few lemmas first. Note that there is no restriction on the characteristic of the field $K$ in this section. 

%

\begin{Lemma}  \label{decompose}
	Let $S = K[x_1,\ldots,x_r]$ be a polynomial ring in $r$ variables over a field $K$ and let $J = (x_1^{v_1},\ldots,x_r^{v_r})$ be an ideal of $S$ such that $v_i >0$ for all $i.$ Let $\pp_1, \ldots, \pp_\alpha$, for $\alpha \geq 2$, be distinct $S$-ideals generated by subsets of $\{x_1,\ldots,x_r\}.$ Let $I = \cap_{i=1}^{\alpha} \ \pp_i.$ Then for $q,n \in \NN$,
	\begin{align} 
		\ell\left(\frac{S}{I+(J^{[q]})^n}\right) 
		&= \sum_{i=1}^{\alpha} \ell\left(\frac{S}{\pp_i+(J^{[q]})^n}\right) - \sum_{1 \leq i<j \leq \alpha} \ell\left(\frac{S}{\pp_i+\pp_j+(J^{[q]})^n}\right)+ \cdots \nonumber\\
		&\hspace{7cm}+(-1)^{\alpha-1} \ell\left(\frac{S}{\sum_{i=1}^{\alpha} \pp_i+(J^{[q]})^n}\right).
	\end{align}
\end{Lemma}

\begin{proof}
	The result follows using the same arguments as in the proof of \cite[Corollary 3.3]{bgv}.
\end{proof}

%

\begin{Lemma}  \label{polypower}
	Let $S = K[x_1,\ldots,x_r]$ be a polynomial ring in $r$ variables over a field $K$ and let $J = (x_1^{v_1},\ldots,x_r^{v_r})$ where $v_i >0$ for all $i.$ Then for all $q,k \in \NN,$
	\[ \ell \left(\frac{S}{(J^{[q]})^k}\right) =q^r \binom{k+r-1}{r}  \prod_{i=1}^{r} v_i. \]
\end{Lemma}

\begin{proof}
	Since $S$ is a polynomial ring, $J$ and hence $J^{[q]}$ is an ideal generated by a homogeneous system of parameters in $S.$ Using the Cohen-Macaulay property of $S$, it follows that for all $q,k \in \NN,$
	\[ \ell \left(\frac{S}{(J^{[q]})^k}\right) =q^r \binom{k+r-1}{r}  \prod_{i=1}^{r} v_i. \]
	
\end{proof}

%

\begin{Theorem}  \label{polydelta}
	Let $R$ be a $d$-dimensional face ring of a simplicial complex. Let $J = (x_1^{v_1},\ldots,x_r^{v_r})R$ where $v_i >0$ for all $i.$ Then the generalized Hilbert-Kunz function $\ell(R/(J^{[q]})^k)$ is a polynomial for all $q,k \in \NN.$
\end{Theorem}

\begin{proof}
	Write $R = S/I$, where $S = K[x_1,\ldots,x_r]$ and $I = \cap_{i=1}^{\alpha} \pp_i.$ Then for all $q,k \in \NN,$
	\begin{align}  \label{sumpoly}
		\ell\left(\frac{R}{(J^{[q]})^k}\right) 
		&= \ell\left(\frac{S}{I + (J^{[q]})^k}\right) \nonumber\\
		&= \sum_{i=1}^{\alpha} \ell\left(\frac{S}{\pp_i+(J^{[q]})^k}\right) - \sum_{1 \leq i<j \leq \alpha} \ell\left(\frac{S}{\pp_i+\pp_j+(J^{[q]})^k}\right)+ \cdots + (-1)^{\alpha-1} \ell\left(\frac{S}{\sum_{i=1}^{\alpha} \pp_i+(J^{[q]})^k}\right),
	\end{align}
	using Lemma \ref{decompose}. The result follows using Lemma \ref{polypower}.
\end{proof}

From Theorem \ref{polydelta}, it follows that for all $q \geq 1,$ the Hilbert-Samuel function of the ideal $J^{[q]}$ coincides its Hilbert-Samuel polynomial for all $k \geq 1.$ In particular, we can write, for all $q, k \geq 1,$
\begin{align}  \label{hilbertPoly}
	\ell \left(\frac{R}{(J^{[q]})^k}\right) = \sum_{i=0}^{d} (-1)^i e_i(J^{[q]}) \binom{k+d-1-i}{d-i}.
\end{align}

%
%

\begin{Theorem}  \label{strings}
	Let $R$ be a $d$-dimensional face ring of a simplicial complex and let $J = (x_1^{v_1},\ldots,x_r^{v_r})$ be an ideal of $R$ such that $v_i >0$ for all $i.$ Then $\lim\limits_{q \rightarrow \infty} e_i(J^{[q]})/q^d =0$ for all $i=1,\ldots,d$ and for all $k \geq 1,$ the generalized Hilbert-Kunz multiplicity
	\[ e_{HK}(J^k) = e_0(J) \binom{k+d-1}{d}. \]
\end{Theorem}

\begin{proof}  
	(i) Comparing the equations \eqref{hilbertPoly} and \eqref{sumpoly} and using Lemma \ref{polypower}, it follows that for all $i=1,\ldots,d,$ the coefficients $e_i(J^{[q]})$ are of the form $\beta_i q^{d-i}$, for some $\beta_i \in \ZZ.$ In particular, for all $i=1,\ldots,d$, 
	\[ \lim\limits_{q \rightarrow \infty} \frac{e_i(J^{[q]})}{q^d} = \lim\limits_{q \rightarrow \infty} \frac{\beta_i q^{d-i}}{q^d} = 0. \] 
	
	(ii) In \eqref{hilbertPoly}, divide by $q^d$ and take limit $q \rightarrow \infty.$ Using part (i) it follows that for all $k \geq 1,$
	\[ e_{HK}(J^k) = \sum_{i=0}^{d} (-1)^i \binom{k+d-1-i}{d-i} \ \lim\limits_{q \rightarrow \infty} \frac{e_i(J^{[q]})}{q^d} = e_0(J) \binom{k+d-1}{d}. \]
\end{proof}

This answers the questions (\cite[Question 1.2, Question 1.3]{ilya}) posed by I. Smirnov for ideals, primary for the maximal ideal, generated by pure powers of the variables and in particular, for the maximal homogeneous ideal of face rings of simplicial complexes. 

Let $J = (x_1^{v_1},\ldots,x_r^{v_r})$ be an ideal of $R=K[\Delta]$ such that $v_i >0$ for all $i$ and let $\nn$ denote the maximal homogeneous ideal of $R.$ From Lemma \ref{polypower} and Theorem \ref{polydelta} it follows that for all $q,k \geq 1,$
\[ \ell\left(\frac{R}{(J^{[q]})^k}\right) = \ell\left(\frac{R}{(\nn^{[q]})^k}\right) \prod_{i=1}^r v_i. \]
Therefore, we only compute the generalized Hilbert-Kunz function of $\nn^k$ in the following  examples. We use the following result to check the stability of an ideal.

\begin{Theorem}[{\cite[Theorem 2.1]{huneke1987}, \cite[Theorem 3.3]{ooishiCoeff}}]  \label{H-O}
	For any $\mm$-primary ideal $I$ of a Cohen-Macaulay local ring $A$, $I$ is a stable ideal if and only if $e_1(I) = e_0(I) - \ell(A/I).$
\end{Theorem}

\begin{Example}  \label{pathSC}
	{\rm
	Let $\Delta$ be the $1$-dimensional simplicial complex on $r$ vertices, for some $r \geq 3:$
	\medskip
	\begin{center}
		\begin{tikzpicture}
			\draw (0,0) -- (2.5,0);
			\filldraw (0,0) circle (2pt) node[anchor=north]{$x_1$};
			\filldraw (1,0) circle (2pt) node[anchor=north]{$x_2$};
			\filldraw (2,0) circle (2pt) node[anchor=north]{$x_3$};
			\draw [dashed] (3,0) -- (4,0);
			\draw (4.5,0) -- (6,0);
			\filldraw (5,0) circle (2pt) node[anchor=north]{$x_{r-1}$};
			\filldraw (6,0) circle (2pt) node[anchor=north]{$x_r$};
		\end{tikzpicture}
	\end{center}
	
	For $i=1,\ldots,r-1$, set $\pp_i = \big( \{x_1,\ldots,x_r\} \setminus \{x_i, x_{i+1}\} \big).$ Let $I = \cap_{i=1}^{r-1} \ \pp_i.$ Then $R=K[x_1,\ldots,x_r]/I$ is the face ring of $\Delta.$ As $\Delta$ is a shellable simplicial complex, it follows that $R$ is a two-dimensional Cohen-Macaulay ring with $f$-vector $f(\Delta) = (1,r,r-1)$ (see \cite[Definition 5.1.11, Theorem 5.1.13]{brunsHerzog}). Let $\nn$ be the maximal homogeneous ideal of $R$. From Theorem \ref{polydelta} it follows that for all $q, k \geq 1$,
	\begin{align*}
		\ell\left(\frac{R}{\nn^{[q]k}}\right) 
		=& \ \ell\left(\frac{S}{I+\mm^{[q]k}}\right)  \\
		=& \ \sum_{i=1}^{r-1} \ell\left(\frac{S}{\pp_i+\mm^{[q]k}}\right) - \sum_{1 \leq i < j \leq r-1} \ell\left(\frac{S}{\pp_i + \pp_j + \mm^{[q]k}}\right) + \cdots + (-1)^{r-2} \ \ell\left(\frac{S}{\sum_{i=1}^{r-1}\pp_i+\mm^{[q]k}}\right).
	\end{align*}
	For $1 \leq i < j \leq r-1$, if $\{x_i,x_{i+1}\} \cap \{x_{j}, x_{j+1}\} \neq \emptyset$, then $S/(\pp_i+\pp_j) \simeq K{[x]}$, i.e. it is a polynomial ring over $K$ in one variable and there are $r-2$ such instances. Otherwise, $S/(\pp_i+\pp_j) \simeq K.$ Therefore, using Lemma \ref{polypower}, we get
	\[ \sum_{1 \leq i < j \leq r-1} \ell\left(\frac{S}{\pp_i + \pp_j + \mm^{[q]k}}\right) = (r-2)kq - \left[ \binom{r-1}{2} - (r-2) \right]. \]
	It is also easy to observe that $S/(\pp_{i_1}+\cdots+\pp_{i_u}) \simeq K,$ for all $u \geq 3$ and $i_1,\ldots,i_u \in \{1,\ldots,r-1 \}.$ Hence, for $q,k \geq 1,$
	\begin{align}  \label{pathSCeq}
		\ell\left(\frac{R}{\nn^{[q]k}}\right) 
		=& \ (r-1) \binom{k+1}{2} q^2 - (r-2)kq - \left[ \binom{r-1}{2} - (r-2) \right] + \binom{r-1}{3} + \cdots + (-1)^{r-2}   \nonumber\\
		=& \ (r-1) \binom{k+1}{2} q^2 - (r-2)kq,
	\end{align}
	where the last equality follows as for any $n \in \NN,$ $\sum_{i=2}^{n} (-1)^i \binom{n}{i} = n-1.$
	This implies that 
	\[ e_0(\nn^{[q]}) = (r-1)q^2, \ e_1(\nn^{[q]}) = (r-2)q, \ e_2(\nn^{[q]}) = 0 \text{ and } e_{HK}(\nn^k) = (r-1)\binom{k+1}{2}.\] 
	Since $e_1(\nn^{[q]}) = e_0(\nn^{[q]}) - \ell(R/\nn^{[q]})$, from Theorem \ref{H-O} it follows that $r(\nn^{[q]}) = 1$ and hence $G(\nn^{[q]})$ is Cohen-Macaulay for all $q.$
}\end{Example}

\begin{Example}  \label{cycleSC}
	{\rm
	Let $\Delta$ be the the simplicial complex related to an $n$-cycle and $S=K [x_1, \ldots, x_n]$. Let $I=I_{\Delta}, R=S/I$. Let $\nn$ be the maximal homogeneous ideal of $R$ and $\mm=(x_1, \ldots, x_n)S$. 
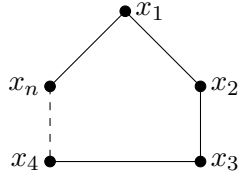
	\begin{center}
		\begin{tikzpicture}
			\draw (0,0) -- (1,-1);
			\draw (1,-1) -- (1,-2);
			\draw (-1,-2) -- (1,-2);
			\draw[dashed] (-1,-1) -- (-1,-2);
			\draw (-1,-1) -- (0,0);
			\filldraw (0,0) circle (2pt) node[anchor=west]{$x_1$};
			\filldraw (1,-1) circle (2pt) node[anchor=west]{$x_2$};
			\filldraw (1,-2) circle (2pt) node[anchor=west]{$x_3$};
			\filldraw (-1,-2) circle (2pt) node[anchor=east]{$x_4$};
			\filldraw (-1,-1) circle (2pt) node[anchor=east]{$x_n$};	
		\end{tikzpicture}
	\end{center}

	For $i=1,\ldots,n-1$, set $\pp_i = \big( \{x_1,\ldots,x_n\} \setminus \{x_i, x_{i+1}\} \big)$ and $\pp_n = (x_2,\ldots,x_{n-1}).$ Then $I = \cap_{i=1}^{n} \ \pp_i.$ From Theorem \ref{polydelta} it follows that for all $q,k \geq 1,$
	\begin{align*}
		\ell\left(\frac{R}{\nn^{[q]k}}\right) 
		= \sum_{i=1}^{n} \ell\left(\frac{S}{\pp_i+\mm^{[q]k}}\right) - \sum_{1 \leq i < j \leq n} \ell\left(\frac{S}{\pp_i + \pp_j + \mm^{[q]k}}\right) + \cdots + (-1)^{n-1} \ \ell\left(\frac{S}{\sum_{i=1}^{n}\pp_i+\mm^{[q]k}}\right).
	\end{align*}
	Using the same arguments as in the above example, it follows that
	\begin{align}  \label{cycle}
		\ell\left(\frac{R}{\nn^{[q]k}}\right) 
		=& \ n \binom{k+1}{2} q^2 - nkq - \left[ \binom{n}{2} - n \right] + \binom{n}{3} + \cdots + (-1)^{n-1}   \nonumber\\
		=& \ n \binom{k+1}{2} q^2 - nkq +1.
	\end{align}
	This implies that for all $k \geq 1,$ 
	\[ e_0(\nn^{[q]}) = nq^2, \ e_1(\nn^{[q]}) = nq, \ e_2(\nn^{[q]}) = 1 \text{ and } e_{HK}(\nn^k) = n\binom{k+1}{2}.\] 
	We now prove that $\depth G(\nn^{[q]}) \geq 1$ for all $q.$ 
%
	We use the following result by M.E. Rossi and G. Valla.
	\begin{Theorem}[{\cite[Corollary 1.7]{rossi}}] \label{rossi-valla}
		Let $I$ be an $\mm$-primary ideal of a Cohen-Macaulay local ring $(R, \mm)$ of dimension $d$. If $e_0(I)= \ell(I/I^2) + (1-d) \ell(R/I)+1,$ then $\depth G(I) \geq d-1.$
	\end{Theorem}
		
%
		
	From equation \eqref{cycle}, we get that for all $k \geq 1,$
	\[ \ell\left(\frac{(\nn^{[q]})^k}{{(\nn^{[q]})^{k+1}}}\right)
	= nq^2{{k+2}\choose 2}-nq(k+1)+1 -nq^2{{k+1}\choose 2}+nqk-1 =nq^2(k+1)-nq. \]
	
	Hence 
	\[ \ell(\nn^{[q]}/(\nn^{[q]})^2)+ (1-d)\ell(R/\nn^{[q]}) + 1 = 2nq^2-nq-(nq^2-nq+1)+1 = nq^2 = e_0(\nn^{[q]}).\] 
	Using Theorem \ref{rossi-valla}, we get depth $G(\nn^{[q]}) \geq 1$ for all $q.$ 
%
}\end{Example}

Note that Theorem \ref{polydelta} and Theorem \ref{strings} also hold for edge ideals of graphs. 
\begin{Definition}
	Let $G$ be a finite simple graph with vertices $V=V(G)=\{x_1, \ldots, x_n\}$ and the edges $E=E(G)$. The edge ideal $I(G)$ of $G$ is defined to be the ideal in $K[x_1,\ldots,x_n]$ generated by the square free quadratic monomials representing the edges of $G$, i.e., 
	\[ I(G) = \langle x_ix_j \mid x_ix_j \in E \rangle. \]
\end{Definition}

A vertex cover of a graph is a set of vertices such that every edge has at least one vertex belonging to that set. A minimal vertex cover is a vertex cover such that none of its subsets is a vertex cover. For any graph $G$ with the set of all minimal vertex covers $C$, the edge ideal $I(G)$ has the primary decomposition: 
\[ I(G)= \bigcap_{\{x_{i_1}, \ldots, x_{i_u}\}\in C} (x_{i_1}, \ldots, x_{i_u}). \]

For example, when $G$ is a five cycle, the primary decomposition of the edge ideal $I(G)$ is 
\[ I(G) = (x_1x_2, x_2x_3, x_3x_4, x_4x_5, x_5x_1) = (x_1,x_2,x_4) \cap (x_1,x_3,x_5) \cap (x_1,x_3,x_4) \cap (x_2,x_3,x_5) \cap (x_2,x_4,x_5). \]

\begin{Example}{\rm
	Let $\Delta$ be the simplicial complex
	
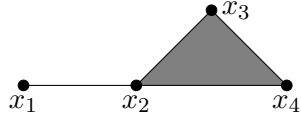
\begin{center}
		\begin{tikzpicture}
			\draw (0,0) -- (1.5,0);
			\filldraw[fill=gray, draw=black] (1.5,0) -- (2.5,1) -- (3.5,0) -- cycle;
			\filldraw (0,0) circle (2pt) node[anchor=north]{$x_1$};		
			\filldraw (1.5,0) circle (2pt) node[anchor=north]{$x_2$};
			\filldraw (2.5,1) circle (2pt) node[anchor=west]{$x_3$};
			\filldraw (3.5,0) circle (2pt) node[anchor=north]{$x_4$};
		\end{tikzpicture}
	\end{center}
	Then $R = K[x_1,x_2,x_3,x_4]/((x_1) \cap (x_3,x_4))$ is the face ring of $\Delta.$ Observe that $R$ is not a Cohen-Macaulay ring. Let $I=(x_1) \cap (x_3, x_4)$. Here the defining ideal is the edge ideal of a line graph with three vertices. Set $S=K[x_1, x_2, x_3, x_4],$ $\pp_1 = (x_1),$ and $\pp_2 = (x_3,x_4).$ Therefore, for $q,k \geq 1,$
	\begin{align*}
		\ell\left(\frac{R}{\nn^{[q]k}}\right) 
		= \ell\left(\frac{S}{I+\mm^{[q]k}}\right)  
		=& \ \ell\left(\frac{S}{\pp_1+\mm^{[q]k}}\right) + \ell\left(\frac{S}{\pp_2 + \mm^{[q]k}}\right) - \ell\left(\frac{S}{\pp_1 + \pp_2 + \mm^{[q]k}}\right)  \\[1mm]
		=& \ \binom{k+2}{3} q^3 + \binom{k+1}{2} q^2 - kq.
	\end{align*}
	This implies that $e_0(\nn^{[q]}) = q^3,$ $e_1(\nn^{[q]}) = -q^2,$ $e_2(\nn^{[q]}) = -q,$ $e_3(\nn^{[q]})=0$ and $e_{HK}(\nn^k) = \binom{k+2}{3}.$
}\end{Example}

\begin{Example}[\bf Complete Bipartite Graphs]{\rm
	A complete bipartite graph $K_{\alpha,\beta}$ is a graph whose set of  vertices is decomposed into two disjoint sets such that no two  vertices within the same set are adjacent and that every pair of  vertices in the two sets are adjacent.
	
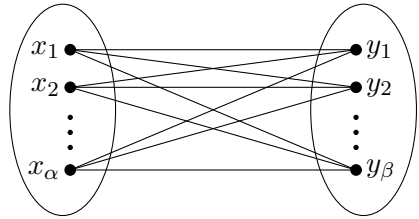
\begin{align*}
		\begin{tikzpicture}
			\draw (0,0) ellipse (20pt and 40pt);
			\draw (4,0) ellipse (20pt and 40pt);	
			\filldraw (0.1,.8) circle (2pt) node[anchor=east]{$x_1$};
			\filldraw (0.1,.3) circle (2pt) node[anchor=east]{$x_2$};	
			\filldraw (0.1,-.1) circle (.8pt);	
			\filldraw (0.1,-.3) circle (.8pt);
			\filldraw (0.1,-.5) circle (.8pt);
			\filldraw (0.1,-.8) circle (2pt) node[anchor=east]{$x_\alpha$};	
			\filldraw (3.9,.8) circle (2pt) node[anchor=west]{$y_1$};
			\filldraw (3.9,.3) circle (2pt) node[anchor=west]{$y_2$};	
			\filldraw (3.9,-.1) circle (.8pt);	
			\filldraw (3.9,-.3) circle (.8pt);
			\filldraw (3.9,-.5) circle (.8pt);
			\filldraw (3.9,-.8) circle (2pt) node[anchor=west]{$y_\beta$};	
			\draw (.1,.8) -- (3.9,.8);
			\draw (.1,.8) -- (3.9,.3);
			\draw (.1,.8) -- (3.9,-.8);
			\draw (.1,.3) -- (3.9,.8);
			\draw (.1,.3) -- (3.9,.3);
			\draw (.1,.3) -- (3.9,-.8);
			\draw (.1,-.8) -- (3.9,.8);
			\draw (.1,-.8) -- (3.9,.3);
			\draw (.1,-.8) -- (3.9,-.8);
		\end{tikzpicture}
	\end{align*}
	Let $S=K[x_1, \ldots, x_\alpha, y_1, \ldots, y_\beta],$  where $1 \leq \alpha \leq \beta.$ Then the edge ideal $I = I(K_{\alpha,\beta}) = \big(x_i y_j \mid 1 \leq i \leq \alpha, 1\leq j \leq \beta \big)$ and $R=S/I$ is a $\beta$-dimensional ring. Let  $\pp_1 = (x_1, \ldots, x_\alpha)$, $\pp_2 = (y_1, \ldots, y_\beta).$ Thus $I=\pp_1 \cap \pp_2$. Note that $I$ is the Stanley-Reisner ideal of the union of an $\alpha$-simplex and a $\beta$-simplex. We compute the generalized Hilbert-Kunz function of $\nn^k$ for different cases of $\alpha$ and $\beta.$ Recall that for all $q,k \geq 1,$ 
	\[ \ell\left(\frac{R}{\nn^{[q]k}}\right) 
	= \ell\left(\frac{S}{I+\mm^{[q]k}}\right)  
	= \ \ell\left(\frac{S}{\pp_1+\mm^{[q]k}}\right) + \ell\left(\frac{S}{\pp_2 + \mm^{[q]k}}\right) - \ell\left(\frac{S}{\pp_1 + \pp_2 + \mm^{[q]k}}\right). \]
	{\bf Case 1}: Let $\beta = \alpha.$ Then for $q,k \geq 1,$
	\begin{align}  \label{BGeq}			
		\ell\left(\frac{R}{\nn^{[q]k}}\right) 
		= 2\binom{k+\beta-1}{\beta} q^\beta - 1.
	\end{align}
	This implies that $e_0(\nn^{[q]}) = 2q^\beta,$ $e_{\beta}(\nn^{[q]}) = (-1)^{\beta+1}$ and $e_{HK}(\nn^k) = 2\binom{k+\beta-1}{\beta}.$		
	
	{\bf Case 2}: Let $\beta \geq \alpha+1.$ Then for $q,k \geq 1,$
	\begin{align*}			
		\ell\left(\frac{R}{\nn^{[q]k}}\right) 
		= \binom{k+\beta-1}{\beta} q^\beta + \binom{k+\alpha-1}{\alpha} q^\alpha - 1.
	\end{align*}
	This implies that $e_0(\nn^{[q]}) = q^\beta,$ $e_{\beta-\alpha}(\nn^{[q]}) = (-1)^{\beta-\alpha} q^\alpha$, $e_{\beta}(\nn^{[q]}) = (-1)^{\beta+1}$ and $e_{HK}(\nn^k) = \binom{k+\beta-1}{\beta}.$
}\end{Example}

\section{Counter-example to Smirnov's conjecture}

Smirnov proposed that following conjecture in \cite{ilya}.
\begin{Conjecture}  \label{conjecture}
	Let $(R,\mm)$ be a Cohen-Macaulay local ring. Then an $\mm$-primary ideal $I$ is stable if and only if $\lim\limits_{q \rightarrow \infty}	e_1(I^{[q]})/q^d = e_0(I) - e_{HK}(I).$
\end{Conjecture}

From section 3, we know that if $R$ is a $d$-dimensional face ring of a simplicial complex and $J = (x_1^{v_1},\ldots,x_r^{v_r})$ is an $R$-ideal such that $v_i >0$ for all $i$, then $\lim\limits_{q \to \infty} e_1(J^{[q]})/q^d=0$ and $e_{HK}(J) = e_0(J).$ This means that $\lim\limits_{q \rightarrow \infty}	e_1(J^{[q]})/q^d = e_0(J) - e_{HK}(J)$ is always true for such ideals. In particular, the conjecture is false whenever $J$ is not a stable ideal. The following example illustrates one such case.


\begin{Example}{\rm 
	Let $\Delta$ be the simplicial complex 
	
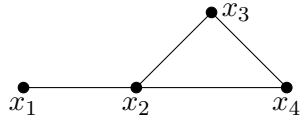
\begin{center}
		\begin{tikzpicture}
			\draw (0,0) -- (1.5,0);
			\draw (1.5,0) -- (2.5,1) -- (3.5,0) -- cycle;
			\filldraw (0,0) circle (2pt) node[anchor=north]{$x_1$};		
			\filldraw (1.5,0) circle (2pt) node[anchor=north]{$x_2$};
			\filldraw (2.5,1) circle (2pt) node[anchor=west]{$x_3$};
			\filldraw (3.5,0) circle (2pt) node[anchor=north]{$x_4$};
		\end{tikzpicture}
	\end{center}
	Then $R = K[x_1,x_2,x_3,x_4]/((x_3,x_4) \cap (x_1,x_3) \cap (x_1,x_4) \cap (x_1,x_2))$ is the face ring of $\Delta.$ Observe that $R$ is a $2$-dimensional Cohen-Macaulay ring with $f$-vector $f(\Delta) = (1,4,4).$ Let $\nn = (x_1,x_2,x_3,x_4)$ denote the maximal ideal of $R.$ Using Theorem \ref{strings}, it follows that $4 = e_0(\nn) = e_{HK}(\nn)$ and $\lim\limits_{q \rightarrow \infty} e_1(\nn^{[q]})/q^2=0.$ We prove that $\nn$ is not stable. Consider the Hilbert series of $R$,
	\[ H(R,z) = 1 + \frac{4z}{1-z} + \frac{4z^2}{(1-z)^2} = \frac{1+2z+z^2}{(1-z)^2}. \]
	As $R$ is Cohen-Macaulay and degree of the numerator above is 2, we get $r(\nn^{[q]})=2.$
	
	We also observe that in this case, $\depth G(\nn^{[q]}) \geq 1.$ Set $\pp_1 = (x_3, x_4),$ $\pp_2 = (x_1, x_3),$ $\pp_3 = (x_1, x_4)$ and $\pp_4 = (x_1, x_2).$ Then $I = \cap_{i=1}^{4} \ \pp_i$ and $R = S/I,$ where $S = K[x_1,\ldots,x_4].$ From Theorem \ref{polydelta} it follows that for all $q,k \geq 1,$
	\begin{align} \label{counter}
		\ell\left(\frac{R}{\nn^{[q]k}}\right) 
		=& \ \sum_{i=1}^{4} \ell\left(\frac{S}{\pp_i+\mm^{[q]k}}\right) - \sum_{1 \leq i < j \leq 4} \ell\left(\frac{S}{\pp_i + \pp_j + \mm^{[q]k}}\right) + \sum_{1 \leq i < j < u \leq 4} \ell\left(\frac{S}{\pp_i + \pp_j + \pp_u + \mm^{[q]k}}\right) \nonumber\\
		&\hspace{8.3cm}- \ell\left(\frac{S}{\sum_{i=1}^{4}\pp_i+\mm^{[q]k}}\right)  \nonumber\\ 
		=& \ 4 \binom{k+1}{2} q^2 - (5kq + 1) + (kq + 3) - 1   \nonumber\\
		=& \ 4 \binom{k+1}{2} q^2 - 4kq +1.
	\end{align}
	This implies that 
	\[ e_0(\nn^{[q]}) = 4q^2, \ e_1(\nn^{[q]}) = 4q, \ e_2(\nn^{[q]}) = 1 \text{ and } e_{HK}(\nn^k) = 4 \binom{k+1}{2}.\] 
	From equation \eqref{counter}, we get that for all $k \geq 1,$
	\[ \ell\left(\frac{(\nn^{[q]})^k}{{(\nn^{[q]})^{k+1}}}\right)
	= 4q^2{{k+2}\choose 2}-4q(k+1)+1 -4q^2{{k+1}\choose 2}+4qk-1 =4q^2(k+1)-4q. \]
	
	Hence 
	\[ \ell(\nn^{[q]}/(\nn^{[q]})^2)+ (1-d)\ell(R/\nn^{[q]}) + 1 = 8q^2-4q-(4q^2-4q+1)+1 = 4q^2 = e_0(\nn^{[q]}).\] 
	Using Theorem \ref{rossi-valla}, we get depth $G(\nn^{[q]}) \geq 1$ for all $q.$ 
%
}\end{Example}

\begin{Remark}
	In the above example, $G(\nn) \simeq R$ is Cohen-Macaulay and $\depth G(\nn^{[q]}) \geq d-1$ for all $q.$ This indicates that the conjecture is false even if $\depth G(\nn^{[q]}) \geq d-1.$
\end{Remark}

The above observation indicates that Conjecture \ref{conjecture} does not help in characterizing the stability of ideals in rings, as in the setup of section 3. We end the section by characterizing the stability of ideals in certain classes of Cohen-Macaulay simplicial complexes and graphs. We use Theorem \ref{H-O} for the same.

\begin{Examples}{\rm
(1) Let $R$ be the face ring of the simplicial complex on $r$ vertices, $r \geq 3,$ as in Example \ref{pathSC}. 
\begin{center}
	\begin{tikzpicture}
		\draw (0,0) -- (2.5,0);
		\filldraw (0,0) circle (2pt) node[anchor=north]{$x_1$};
		\filldraw (1,0) circle (2pt) node[anchor=north]{$x_2$};
		\filldraw (2,0) circle (2pt) node[anchor=north]{$x_3$};
		\draw [dashed] (3,0) -- (4,0);
		\draw (4.5,0) -- (6,0);
		\filldraw (5,0) circle (2pt) node[anchor=north]{$x_{r-1}$};
		\filldraw (6,0) circle (2pt) node[anchor=north]{$x_r$};
	\end{tikzpicture}
\end{center}
By putting $q=1$ in \eqref{pathSCeq}, we get for all $k \geq 1,$
\[ \ell(R/\nn^k) = (r-1) \binom{k+1}{2} - (r-2)k. \]
Hence, 
\[ e_0(\nn) - \ell(R/\nn) = (r-1) - \big( r-1 - (r-2) \big) = r-2 = e_1(\nn), \]
implying that $\nn$ is a stable ideal. Let $J = (x_1^{v_1},\ldots,x_r^{v_r})$ be an ideal of $R$ such that $v_i >0$ for all $i.$ Since $\ell(R/(J^{[q]})^k) = \ell(R/(\nn^{[q]})^k) \prod_{i =1}^r v_i,$ it follows that $J$ is a stable ideal.

(2) Let $R$ be the face ring of the simplicial complex related to an $n$-cycle, as in Example \ref{cycleSC}. By putting $q=1$ in \eqref{cycle}, we get for all $k \geq 1,$
\[ \ell(R/\nn^k) = n \binom{k+1}{2} - nk + 1. \]
Hence, 
\[ e_0(\nn) - \ell(R/\nn) = n - 1 \neq e_1(\nn), \]
implying that $\nn$ is not a stable ideal. Similarly, $J$ is not a stable ideal, where $J = (x_1^{v_1},\ldots,x_r^{v_r})$ such that $v_i >0$ for all $i.$ 
}\end{Examples}

\begin{Theorem}
	Consider a complete graph, $K_r,$ on $r$ vertices, where $r \geq 1.$ Let $S = K[x_1,\ldots,x_r]$ and $I = I(K_r)$ denote the edge ideal of $K_r.$ Let $R=S/I$ and $J = (x_1^{v_1},\ldots,x_r^{v_r})$ be an ideal of $R$ such that $v_i >0$ for all $i.$ Then $J$ is a stable ideal.
\end{Theorem}

\begin{proof}
	We begin by observing that any minimal vertex cover of $K_r$ contains exactly $r-1$ vertices. Hence, $I = \cap_{i=1}^r (x_1,\ldots,\hat{x_i},\ldots,x_r).$ Set $\pp_i = (x_1,\ldots,\hat{x_i},\ldots,x_r),$ for all $i=1,\ldots,r.$ Let $\nn$ denote the maximal homogeneous ideal of $R.$ Note that $\pp_i+\pp_j = \nn$ for all $i \neq j.$ Since $\ell(R/(J^{[q]})^k) = \ell(R/(\nn^{[q]})^k) \prod_{i =1}^r v_i,$ we compute $\ell(R/(\nn^{[q]})^k).$ From Theorem \ref{polydelta}, it follows that for all $q,k \geq 1,$
	\begin{align*}
		\ell\left(\frac{R}{\nn^{[q]k}}\right) 
		=& \ \sum_{i=1}^{r} \ell\left(\frac{S}{\pp_i+\mm^{[q]k}}\right) - \sum_{1 \leq i < j \leq r} \ell\left(\frac{S}{\pp_i + \pp_j + \mm^{[q]k}}\right) + \cdots + (-1)^{r-1} \ \ell\left(\frac{S}{\sum_{i=1}^{r}\pp_i+\mm^{[q]k}}\right) \\
		=& \ rkq - \binom{r}{2} + \binom{r}{3} + \cdots + (-1)^{r-1} \\
		=& \ rkq - (r-1).
	\end{align*}
	This implies that $\ell(R/(J^{[q]})^k) = rkq \prod_{i =1}^r v_i - (r-1)\prod_{i =1}^r v_i.$ In particular, $\ell(R/J^k) = rk\prod_{i =1}^r v_i - (r-1)\prod_{i =1}^r v_i.$ Hence, 
	\[ e_0(J) - e_1(J) = r \prod_{i =1}^r v_i - (r-1)\prod_{i =1}^r v_i = \ell(R/J), \]
	implying that $J$ is a stable ideal.
\end{proof}

%
%

\bibliographystyle{plain}
\bibliography{ref}

\end{document}